\documentclass[11pt,draft]{amsart}
\usepackage{amsmath,amsfonts,amsthm,amsopn,amssymb,latexsym,soul}
\usepackage{cite}
\usepackage{color,enumitem,graphicx,esint}
\usepackage[colorlinks=true,urlcolor=blue,
citecolor=red,linkcolor=blue,linktocpage,pdfpagelabels,
bookmarksnumbered,bookmarksopen]{hyperref}
\usepackage[english]{babel}
\usepackage[left=2.61cm,right=2.61cm,top=2.72cm,bottom=2.72cm]{geometry}
\usepackage[hyperpageref]{backref}
\usepackage[colorinlistoftodos]{todonotes}

\makeatletter
\providecommand\@dotsep{5}
\def\listtodoname{List of Todos}
\def\listoftodos{\@starttoc{tdo}\listtodoname}
\makeatother

\numberwithin{equation}{section}
\newtheorem{theorem}{Theorem}[section]

\newtheorem{definition}[theorem]{Definition}
\newtheorem{proposition}[theorem]{Proposition}
\newtheorem{remark}[theorem]{Remark}

\newtheorem{question}[theorem]{Question}
\newtheorem{conjecture}[theorem]{Conjecture}

\newcommand{\R}{\mathbb{R}}

\newcommand{\RN}{{\mathbb{R}^N}}
\newcommand{\LN}{{\mathbb{L}^{N+1}}}
\newcommand{\RT}{{\mathbb{R}^3}}

\renewcommand{\le}{\leqslant}
\renewcommand{\ge}{\geqslant}

\renewcommand{\d }{\delta }

\newcommand{\n }{\nabla }
\newcommand{\e }{\epsilon }

\renewcommand{\L}{\mathbb{L}}

\newcommand{\X}{\mathcal{X}}
\newcommand{\U}{\mathcal{U}}
\newcommand{\N}{\mathbb{N}}

\newcommand{\irn }{\int_{\RN}}

\def\bbm[#1]{\mbo\X{\boldmath $#1$}}
\newcommand{\beq }{\begin{equation}}
\newcommand{\eeq }{\end{equation}}

\def\Xint#1{\mathchoice 
  {\XXint\displaystyle\textstyle{#1}}%
  {\XXint\textstyle\scriptstyle{#1}}%
  {\XXint\scriptstyle\scriptscriptstyle{#1}}%
  {\XXint\scriptscriptstyle\scriptscriptstyle{#1}}%
  \!\int} 
\def\XXint#1#2#3{{\setbox0=\hbox{$#1{#2#3}{\int}$} 
  \vcenter{\hbox{$#2#3$}}\kern-.5\wd0}} 
\def\Mint{\Xint -}

\title
{On the regularity of the minimizer of the electrostatic Born-Infeld energy}
\author{D\MakeLowercase{enis} Bonheure, A\MakeLowercase{lessandro} Iacopetti}

\thanks{\emph{Acknowledgements.} Research partially supported by the project ERC Advanced Grant 2013 n.~339958 Complex Patterns for Strongly Interacting Dynamical Systems COMPAT, by Gruppo Nazionale per l'Analisi Matematica, la Pro\-ba\-bi\-li\-t\`a e le loro Applicazioni (GNAMPA) of the Istituto Nazionale di Alta Matematica (INdAM) by FNRS (PDR T.1110.14F and MIS F.4508.14) and by ARC AUWB-2012-12/17-ULB1- IAPAS}
\address[Denis Bonheure]{D\'epartement de math\'ematique, Universit\'e Libre de Bruxelles, Campus de la Plaine - CP214 boulevard du Triomphe, 1050 Bruxelles, Belgium}
\email{denis.bonheure@ulb.ac.be}
\address[Alessandro Iacopetti]{D\'epartement de math\'ematique, Universit\'e Libre de Bruxelles, Campus de la Plaine - CP214 boulevard du Triomphe, 1050 Bruxelles, Belgium}
\email{alessandro.iacopetti@ulb.ac.be}


\begin{document}
\begin{abstract}
We consider the electrostatic Born-Infeld energy
\begin{equation*}
\int_{\R^N}\left(1-{\sqrt{1-|\nabla u|^2}}\right)\, dx -\int_{\R^N}\rho u\, dx, 
\end{equation*}
where $\rho \in L^{m}(\R^N)$ is an assigned charge density, $m \in [1,2_*]$, $2_*:=\frac{2N}{N+2}$, $N\geq 3$. 
We prove that if $\rho \in L^q(\R^N) $ for $q>2N$, the unique minimizer $u_\rho$ is of class $W_{loc}^{2,2}(\R^N)$. Moreover, if the norm of $\rho$ is sufficiently small, the minimizer is a weak solution of the associated PDE 
\begin{equation}\label{eq:BI-abs}
\tag{$\mathcal{BI}$}
-\operatorname{div}\left(\displaystyle\frac{\nabla u}{\sqrt{1-|\nabla u|^2}}\right)= \rho \quad\hbox{in }\mathbb{R}^N,
\end{equation}
with the boundary condition $\lim_{|x|\to\infty}u(x)=0$ 
and it is of class $C^{1,\alpha}_{loc}(\RN)$, for some $\alpha \in (0,1)$.
\end{abstract}

\maketitle

\tableofcontents

\section{Introduction}

It is well known that, up to a suitable choice of the constants, the time independent Maxwell's equations in the vacuum without current density lead to Poisson's equation
\beq\label{eq:3}
-\Delta u= \rho,
\eeq
where $\rho$ represents the charge density. If $\rho$ is a point charge, i.e. $\rho =\d_0$, where $\d_0$ is  a Dirac mass at the origin,  the solution of \eqref{eq:3} is explicitely given by
$u(x)=1/(4\pi|x|)$ and its energy is infinite, that is
\[
\mathcal{E}(u)=\frac{1}{2}\int_{\RT}|\n u|^2 \ dx=+\infty.
\]
Even when $\rho\in L^1(\RT)$, we cannot say, in general, that \eqref{eq:3} admits a solution with finite energy (see e.g. \cite{FOP} for a counterexample). From a physical point of view this means that Maxwell's model violates the principle of finiteness of the energy. To avoid this phenomenon, Born \cite{B1, B2} and later on Born and Infeld \cite{BI, BI2}, proposed a new model based on the modification of Maxwell's Lagrangian density (we refer to \cite[Sect. 1]{BDP} and the references therein for more details). In the electrostatic case,  Born-Infeld theory yields the Lagrangian
\[
\mathcal{L}_{\rm BI} 
= b^2\left(1-\sqrt{1-\frac{|\bf E|^2}{b^2}}\right),
\]
where $\bf E$ is the electrical field, $b$ is a constant having the dimensions of ${e}/{r_0^2}$, $e$ and $r_0$ being respectively the charge and the effective radius of the electron. If $b\to+\infty$ or for fields having small intensities, $\mathcal{L}_{\rm BI}$ reduces to Maxwell's Lagrangian density $\frac12|\bf E|^2$. In the sequel, we take $b=1$ for simplicity and we take $\RN$, $N\geq3$, as ambient space. Remember also that Faraday's law of induction implies that $\bf E$ is a gradient. 

The counterpart of Poisson's equation in the Born-Infeld theory is the PDE
\begin{equation}\label{eq:BI-eq}
Q^-(u)= \rho,
\end{equation}
where the operator $Q^{-}$ is defined, for weakly spacelike functions (see Definition \ref{def:spacelike}), as
\begin{equation}\label{Q-}
Q^{-}(u)=-\operatorname{div}\left(\frac{\nabla u}{\sqrt{1-{|\nabla u|^2}}}\right).
\end{equation}  
Formally, this equation arises as the Euler-Lagrange equation associated with the electrostatic Born-Infeld energy
\begin{equation}\label{eq:functionalBI}
\tag{$\mathcal{BI}$-energy}
I_\rho(u)=\int_{\R^N}\left(1-{\sqrt{1-|\nabla u|^2}}\right)\, dx -\int_{\R^N}\rho u\, dx.
\end{equation}

It is interesting to notice that the operator $Q^-$ naturally appears  in other contexts, as for instance in string theory (see e.g. \cite{Gibb98}) or in classical relativity where it stands for the mean curvature operator in Lorentz-Minkowski space $(\L^{N+1},(\cdot,\cdot)_{\L^{N+1}})$. Several results have been then obtained for the corresponding Plateau's problem as well as for other situations driven  by the operator $Q^{-}$ (see e.g. \cite{AZ, BoIa,Bayard, BDD, BS,CY, GH, Kl1, Kl2, KM, Maw, M, T82}).\\  

\noindent In this paper, we address the question of the regularity of the minimizer of the electrostatic Born-Infeld energy and of the validity of the associated Euler-Lagrange equation \eqref{eq:BI-eq}.
 As a working functional space, we consider the convex set
\begin{equation*}
\X=D^{1,2}(\RN)\cap \{u\in C^{0,1}(\RN) ;\ |\nabla u|_\infty \le 1\},
\end{equation*}
equipped with the norm
\begin{equation*} \|u\|_{\X}:=\left(\int_{\RN}|\nabla u|^2 \ dx\right)^{1/2},
\end{equation*}
where $N\ge 3$ and $D^{1,2}(\RN)$ is the completion of $C_c^\infty (\RN)$ with respect to above norm. We point out that a boundary condition at infinity is encoded in $\X$, namely any $u\in \X$ satisfies $\lim_{|x|\to\infty}u(x)=0$. In addition, this functional space is weakly closed, it embeds continuously into $W^{1,s}(\RN)$ for all $s\in[2^*,\infty)$ and in $L^\infty(\RN)$ (see \cite[Lemma 2.1]{BDP}).

Let $\X^*$ be the dual space of $\X$ and $\langle \cdot, \cdot \rangle$ be the duality pairing between $\X^*$ and $\X$. We recall that $\X^*$ contains Radon measures as for instance superpositions of point charges or $L^1(\RN)$ densities (see \cite{BDP,K}).
Let $\rho \in \X^*$. Then the electrostatic Born-Infeld energy \eqref{eq:functionalBI}
%
is well defined in $\X$.
Since $I_\rho$ is not smooth at the points $u$ such that $|\nabla u|_\infty=1$,  we need to distinguish between the notion of critical point (in a relaxed sense) and the notion of weak solution. Using the tools of non-smooth critical point theory (see e.g. \cite{EkTe,S} and \cite[Sect. 2]{BDP}), we adopt the following definitions. 
\begin{definition}
 We say that $u_\rho\in \X$ is  a {\em critical point in  weak sense} for the functional $I_\rho$ if $0$ belongs to the subdifferential of $I_\rho$ at $u_\rho$.
 \end{definition}

\begin{definition}\label{def:ws}  
We say that $u_\rho\in \X$ is a {\em weak solution} of the electrostatic Born-Infeld equation
\begin{equation}\label{eq:BI}
\tag{$\mathcal{BI}$}
\left\{
\begin{array}{ll}
-\operatorname{div}\left(\displaystyle\frac{\nabla u}{\sqrt{1-|\nabla u|^2}}\right)= \rho \quad \hbox{in }\mathbb{R}^N,
\\[6mm]
\displaystyle\lim_{|x|\to \infty}u(x)= 0 
\end{array}
\right.
\end{equation} 
if for all $\psi \in\X$ we have
\begin{equation}\label{eq:weakBI}
\irn \frac{\n u_\rho \cdot \n \psi}{\sqrt{1-|\nabla u_\rho|^2}}\, dx
=\langle \rho, \psi\rangle. 
\end{equation}
\end{definition}

\noindent We point out that, in our context, the notion of critical point in weak sense is equivalent to ask that $u_\rho$ is a minimum for the functional $I_\rho$ (see \cite[Sect. 2]{BDP}), and, if $\rho$ is a distribution, the weak formulation of \eqref{eq:weakBI} extends to any test function $\psi \in C^{\infty}_{c}(\RN)$.\\

\noindent It is quite standard to show that $I_\rho$ is  bounded from below in $\X$, coercive, weakly lower semi-\, Continuous and strictly convex, and thus by the direct methods of the Calculus of Variations, a minimizer always exist and is unique (see \cite[Proposition 2.3]{BDP}). In addition, any weak solution of \eqref{eq:BI} would coincide with the minimizer (see \cite[Proposition 2.6]{BDP}). Therefore, a first natural question arises.

\begin{question}\label{q1}
If $\rho \in \X^*$, is the minimizer $u_\rho$ always a weak solution of \eqref{eq:BI}? 
\end{question}

\noindent This question has motivated several publications in the past years (see e.g. \cite{BDP,K}) and it seems hard to answer it in full generality under the mere assumption $\rho \in \X^*$. On the other hand, for any given $\rho \in \X^*$  it is possible to show that the singular set 
\beq\label{def:singset}
S=\{x \in \RN; \ |\nabla u_\rho|=1\}
\eeq
has null Lebesgue measure and the  minimizer $u_\rho$ satisfies
\beq\label{ineq1}
\irn \frac{|\nabla u_\rho|^2}{\sqrt{1-|\nabla u_\rho|^2}} \ dx \leq \langle \rho, u_\rho\rangle.
\eeq
If in addition the equality is attained in \eqref{ineq1} then $u_\rho$ is a weak solution (see \cite[Proposition 2.7, Remark 2.8]{BDP}). Proving the equality in \eqref{ineq1} seems a rather difficult task in general. 

Nevertheless, in some special cases the answer to Question \ref{q1} is positive. Indeed, when $\rho \in \X^*$ is radially distributed or when $\rho \in L^{\infty}_{loc}(\RN)\cap \X^*$, the minimizer $u_\rho$ is a weak solution to \eqref{eq:BI} (see \cite[Theorem 1.4, Theorem 1.5]{BDP}). For less regular data, as in the case of superposition of charges, namely $\rho=\sum_{i=1}^k a_i \delta_{x_i}$, where  $a_i \in \R$, $x_i \in \R^N$, $i=1,\ldots,k$,  the minimizer $u_\rho$ is a weak solution away from the charges, i.e. $u_\rho$ weakly solves 
$$ -\operatorname{div}\left(\displaystyle\frac{\nabla u}{\sqrt{1-|\nabla u|^2}}\right)= 0 \quad \hbox{in }\ \R^N \setminus \{x_1,\ldots,x_k\}.$$
Moreover, $u_\rho$ is a classical solution of the same equation in $\RN\setminus\Gamma$, where $\Gamma$ is the set of all segments whose endpoints are the charges $\{x_1,\ldots,x_k\}$. Under further assumptions we can say more: if the intensities $|a_i|$ are sufficiently small then $u_\rho$ is a classical solution in $\R^N \setminus \{x_1,\ldots,x_k\}$, it is of class $C^\infty(\RN\setminus\{x_1,\ldots,x_N\})$, strictly spacelike away from the charges and $\lim_{x\to x_i} |\nabla u_\rho|=1$ for all $i=1,\ldots,k$ (see \cite{K}, \cite[Theorem 1.6]{BDP}, \cite[Theorem 1.2]{BCF}).\\

\noindent  Another interesting and difficult issue concerns the regularity of the minimizer $u_\rho$, when dealing with $L^q$ data, $q\geq 1$ and $q\ne \infty$. 
To our knowledge, the only result available concerns the case of radial functions. Namely, if $\rho \in L_{rad}^q(\RN) \cap \X^*$, $q\geq1$ then $u_\rho \in C^1(\RN\setminus\{0\})$ and if $\rho \in L_{rad}^q(B_\delta(0)) \cap L^s(\R^N)  \cap \X^*$, for $s\geq 1$, $q\geq N$ and some $\delta>0$, then $u_\rho \in C^1(\R^N)$ (see \cite[Theorem 3.2]{BDP}) and $S$ defined by \eqref{def:singset} is empty. The proofs of these statements deeply rely on ODE techniques, and thus we cannot mimic them in the general case.

As already mentioned, when the datum is more regular (radially symmetric or not), as for instance $\rho \in L^\infty_{loc}(\R^N) \cap \X^*$, the minimizer $u_\rho$ is a weak solution to \eqref{eq:BI} and it is locally of class $C^{1,\alpha}$, for some $\alpha \in (0,1)$. Moreover, when $\rho \in C^k(\R^N) \cap \X^*$ then $u_\rho \in C^{k+1}(\RN)$ (see \cite{BS, BDP, COOS}). 
\medbreak

In a recent paper \cite{KUM2013}, Kuusi and Mingione proved a pointwise gradient estimate for weak solutions to the $p$-Laplacian	
$$-{\rm div}(|\nabla u|^{p-2}\nabla u)=\mu \quad\hbox{in }\mathbb{R}^N,$$
for $p \ge 2$, of the kind 
$$|\nabla u(x)|^{p-1} \le c(N, p) {\bf I}^{|\mu|}_1(x,\infty)$$
where 
\beq\label{def:RieszPot}
{\bf I}^{|\mu|}_1(x,\infty):=\int_0^{+\infty}\frac{|\mu|(B_t(x))}{t^N}dt \eeq
denotes the linear Riesz potential of $|\mu|$.
We also refer to \cite{KUM2014,Ming11}. In particular, this shows that if $\mu$ is an $L^q$-datum with $q<N$ which is moreover locally $L^s$ for some $s>N$, then the solution has a bounded gradient.   This gradient estimate remarkably extends the classical gradient estimate that holds for the Poisson equation to the case of the $p$-Laplacian. In the linear case, of course the estimate is a straightforward consequence of the representation formula for the solution, whereas such a simple argument is obviously unavailable in the nonlinear case. As suggested in \cite{BA}, one might see this gradient estimate as a formal inversion of the divergence operator but this obviously has to be turned rigorously into a proof (see \cite{KUM2013}).  Baroni \cite{BA} showed that the principle still holds for a class of  regular quasilinear, possibly degenerate, equations of the form 
$$
-{\rm div}\left(\frac{g(|\nabla u|)}{|\nabla u|}\nabla u\right) =\mu  
$$
in bounded domains, leading to the estimate 
\begin{equation}\label{potentialest}
g(|\nabla u(x)|)\le c {\bf I}_1^{|\mu|}(x,2R)+c g\left(\Mint_{B_R(x)}|\nabla u|dx\right) \,,
\end{equation}
where $c>0$ depends only on $N$ 
and ${\bf I}_1^{|\mu|}(x,2R)$ is the truncated linear Riesz potential (see \cite[Sect. 1]{BA}). We refer to \cite[Theorem 1.2]{BA}  for the precise assumptions on $g$. In the case of singular operators modeled on the $p$-Laplacian, when $2-1/N<p\leq 2$, it is possible to derive analogous estimates in terms of the Riesz potential (see \cite{DM, KUM}).
If this formal inversion of the divergence operator could be justified also for the singular operator such as $Q^-$, one would derive the estimate 
$$\frac{|\nabla u(x)|}{\sqrt{1-|\nabla u(x)|^2}} \le c(N) \int_0^\infty\frac{|\mu|(B_t(x))}{t^N}dt$$
and therefore the norm of the gradient of the solution would stay away from $1$ as soon as $\rho\in L^q$ with $q<N$ and $\rho$ is locally $L^s$ for some $s>N$. We emphasize that in the radial case one can indeed invert the divergence (by integrating the equation), see \cite[Theorem 3.2]{BDP}. The integrability condition on the datum is then sharp as shown by the following example. Consider the toy radial datum $1/|x|^{\beta+1}$ with $\beta>0$. Since the operator $Q^{-}$ acts on radial functions as
$$ Q^-(w)=- r^{1-N}\left(\frac{r^{N-1}w^\prime(r)}{\sqrt{1-w^\prime(r)^2}}\right)^\prime,$$
if 
$$\frac{w^\prime(r)}{\sqrt{1-w^\prime(r)^2}}=-1/r^\beta, \hbox{ for } 0<r<r_0,$$
then $|w'(r)|\to 1$ at the origin and 
$$Q^-(w) = \frac{C}{|x|^{1+\beta}} \hbox{ in } B(0,r_0).$$
Given any $q<N$, we can choose $\beta>0$ small enough so that $q<\frac{N}{\beta+1}$. Then $Q^-(w)$ belongs to $L^q(B(0,r_0))$ but $|w'|$ tends to $1$ at the origin. When $\beta<0$, 
it is easily seen that the solution extends to a $C^{1,\alpha}_{loc}$ function around the origin. 

The analogy with the case of the Poisson equation (and the $p$-Laplacian for $p> 2-1/N$) and the example with the previous radial datum support the following conjecture. 

\begin{conjecture}\label{Q2}
If $\rho \in L^q_{loc}(\RN) \cap \X^*$, with $q>N$, then the minimizer $u_\rho$ is $C^{1,\alpha}_{loc}(\RN)$, for some $\alpha \in (0,1)$.
\end{conjecture}


Differently from the linear framework we do not have a representation formula for the solutions of \eqref{eq:BI}, and we do not even know if the minimizer $u_\rho$ is a weak solution. Proving that the minimizer $u_\rho$ belong to $W^{2,s}_{loc}(\RN)$, for some $s\geq 1$, is already a substantial progress. 
Observe also that the singular operator $Q^-$ does not satisfy the growth and ellipticity conditions of \cite[(1.2)]{KUM}. In particular the so-\, Called ``Nonlinear Calder\'on-Zygmund'' theory does not cover our problem (see \cite{KUM, MING} and the references therein). For these reasons and as it will be clear in the sequel, it is quite challenging to prove Conjecture \ref{Q2}.


Another fascinating way to treat \eqref{eq:BI}, inspired by the papers \cite{FOP,K}, is to see the operator $Q^-$ as a series of $p$-Laplacians, namely

$$Q^-(u)=- \sum_{h=1}^\infty \alpha_h \Delta_{2h}u,$$
where $\alpha_1=\frac{1}{2}$, $\alpha_h=\frac{(2h-3)!!}{2h!!}$ for $h>1 $ (see \cite[(11)]{K}), and $\Delta_{2h}u:= \operatorname{div}\left(|\nabla u|^{2h-2}\nabla u\right)$. Indeed, the operator $Q^-$ is formally the Gateaux derivative of the functional
$$\irn \Big(1 - \sqrt{1-|\nabla u|^2}\Big) dx= \irn  \sum_{h=1}^\infty \frac{\alpha_h}{2h} |\nabla u|^{2h} dx,$$
and if we truncate the expansion up to the order $k$, we obtain a new functional $I_{\rho,k}:\X_{2k} \to \R $,
$$I_{\rho,k}(u):=  \sum_{h=1}^k \frac{\alpha_h}{2h} \irn|\nabla u|^{2h} dx - \langle u, \rho \rangle_{\X_{2k}},$$
where $\X_{2k}$ is the completion of $C^\infty_0(\RN)$ with respect to the norm
$$\|u\|_{\X_{2k}}^2:= \irn |\nabla u|^2 \ dx + \left(\irn |\nabla u|^{2k} \ dx\right)^{1/{2k}}.$$ 
Now the functional $I_{\rho,k}$ is of class $C^1$ and we have existence and uniqueness of a unique critical point for all $k\geq n_0$, provided that $\rho \in \X_{2n_0}^*$ for some $n_0\geq 1$ (see \cite[Proposition 5.1]{BDP}). Moreover, denoting by $u_k$ such a minimizer we have that $u_k$ converges to the unique minimizer $u_\rho$ of $I_\rho$, weakly in $\X_{2m}$ for all $m\geq n_0$ and uniformly on compact sets (see \cite[Theorem 5.2]{BDP}).

When the datum $\rho \in \X_{2n_0}^* \cap L^q(\RN)$, for some $q>N$, we can apply for any $k$ the regularity theory developed for weak solutions of degenerate elliptic equations (see \cite{DB, CM}), which in our specific case are ruled by positive linear combinations of $p$-Laplacians up to the order $2k$. However, we do not have a uniform control on the gradient of $u_k$ and we do not know if the regularity properties of $u_k$ pass to the limit when $k$ goes to infinity. Using this approach seems in fact quite involved. 

\medbreak

The aim of our paper is to give some partial answers to Question \ref{q1} and Conjecture \ref{Q2}. Let $N\geq 3$. We denote by $|\cdot|_q$ the standard norm in $L^q(\mathbb{R}^N)$ and  $2_*:=(2^*)^\prime=\frac{2N}{N+2}$ the conjugate exponent of $2^*=\frac{2N}{N-2}$. 

\begin{theorem}\label{mainteo}
If $\rho\in L^q(\mathbb{R}^N) \cap L^{m}(\RN)$, with $q>2N$, $m \in [1,2_*]$ then $u_\rho \in W^{2,2}_{\rm loc}(\R^N)$.
\end{theorem}

\begin{theorem}\label{mainteo2}
Let $q>2N$ and $m \in [1,2_*]$. There exists a constant $c=c(N,q,m)$ such that if $\rho\in L^q(\mathbb{R}^N)\cap L^{m}(\R^N)$ satisfies  $|\rho|_{q}+|\rho|_{m} \leq c$, then $u_\rho$ is a weak solution of \eqref{eq:BI}, it is strictly spacelike (i.e. $|\n u_\rho|<1$ in $\RN$) and $u \in C^{1,\alpha}_{\rm loc}(\R^N)$, for some $\alpha \in (0,1)$.
\end{theorem}

The proofs of the above results rely on the combination of several tools. First, we consider a standard sequence of mollifiers $(\rho_n)_n$ of the datum $\rho$ and, accordingly, a sequence $(u_n)_n$ of minimizers of \eqref{eq:BI} with data $\rho_n$. The sequence $u_n$ is made of smooth strictly spacelike solutions and converges uniformly to the minimizer $u_\rho$. The second ingredient is a new estimate for smooth solutions of \eqref{eq:BI} with $L^q$ data (see Proposition \ref{prop:mainestimate}) which allows to control the integral of the second derivatives by integral and pointwise quantities associated to the gradient and the $L^q$-norm of the datum. This estimate is inspired by \cite[Lemma 2.1]{BS} and the proof is based on Federer's coarea formula, on the geometry of spacelike hypersurfaces in the Lorentz-Minkowski space, and on a variant of Gronwall's Lemma (see Theorem \ref{variantGrom}). 
At the end, we can prove that the ${W^{2,2}}$-norm of $u_n$ is uniformly bounded in compact subsets of $\RN$ and we easily conclude.\\ 

For the proof of Theorem \ref{mainteo2} we use again the mollification argument, we get suitable energy estimates, and, exploiting the hypotheses together with our gradient estimate, we prove that  $|\nabla u_n|_\infty$ definitely stays away from $1$. Finally, regularizing the operator in a suitable way and applying the estimates of Kuusi and Mingione in terms of the Riesz potential (see \cite[Theorem 1.4]{KUM}) we get that the $C^{1,\alpha}$-norm of $u_n$ in compact subsets of $\RN$ is uniformly bounded, for some $\alpha \in (0,1)$, and the result easily follows.\\ 

We point out that Theorem \ref{mainteo2} cannot be extended to all functions $\rho \in L^q(\mathbb{R}^N)\cap L^{m}(\R^N)$ by a mere scaling argument. In fact, let $t>0$, let $w \in \X$ and set $\tilde w(x):=t w(\frac{x}{t})$, $\tilde \rho(x):=\frac{1}{t} \rho(\frac{x}{t})$. By direct computation we have $|\nabla \tilde w|=|\nabla w|$ and thus $\tilde w \in \X$ for any $t>0$. In addition,
\begin{equation}\label{eq:changelp}
|\tilde\rho|^q_q=t^{N-q}|\rho|_q^q,
\end{equation}
and
\begin{equation}\label{eq:changevarenergy}
I_{\tilde \rho}(\tilde w)=t^N I_\rho(w) \ \forall w \in \X.
\end{equation}
From the uniqueness of the minimizer and \eqref{eq:changevarenergy} it follows that if $u_\rho$ is the minimizer of $I_\rho$ then the minimizer of $I_{\tilde \rho}$ is $\tilde u_\rho$. On the the other hand from \eqref{eq:changelp} it is clear that, if $q>N$, $m \in [1,2_*]$ it is not possible to find a number $t>0$ such that both the $L^q$ and $L^{m}$ norms of $\tilde \rho$ are small so that $|\tilde\rho|_{p}+|\tilde\rho|_{m} \leq c$, where $c$ is the constant given by Theorem \ref{mainteo2}.\\  

\noindent We also stress that our results cannot be recovered directly by known elliptic estimates and gradient bounds (see e.g. \cite{DB, CM, BA, BF, MA1, MA2}). Moreover, we cannot argue as in \cite[Lemma 2.2]{COOS} by putting the equation in non-divergence form because, even if the sequence $(u_n)_n$ is made of smooth strictly spacelike solutions to \eqref{eq:BI} with data $\rho_n$, the sequence of mollified data $\rho_n$ is not necessarily bounded in the $L^\infty$ norm as $n \to +\infty$.\\

The outline of the paper is the following:  in Sect. 2 we recall for convenience some basic facts about the geometry of spacelike cartesian graphs  in the Lorentz-Minkowski space and in Sect. 3 we state and prove the gradient estimate. In Sect. 4 we prove Theorem \ref{mainteo} and in Sect. 6 we prove Theorem \ref{mainteo2}. Finally, we state and prove a new variant of Gronwall's lemma in the Appendix.

\bigbreak

\section{Spacelike vertical graphs in the Lorentz-Minkowski space}

In this section we recall some known facts about the Lorentz-Minkowski space and the geometry of spacelike hypersurfaces which are expressed as cartesian graphs, also called ``vertical graphs''. These result will be useful in the next section in order to derive a gradient estimate for the solutions of \eqref{eq:BI}.

\begin{definition}\label{def:spacelike}
	Let $u\in C^{0,1}(\Omega)$, with $\Omega\subset\RN$. We say that $u$ is
	\begin{itemize}
		\item {\em weakly spacelike} if $|\n u|\le 1$ a.e. in $\Omega$;
		\item {\em spacelike} $|u(x)-u(y)|<|x-y|$ whenever $x,y\in\Omega$, $x\neq y$ and the line segment $\overline{xy}\subset\Omega$;
		\item {\em strictly spacelike} if $u$ is spacelike, $u\in C^1(\Omega)$ and $|\n u|< 1$ in $\Omega$.
	\end{itemize}
\end{definition}

The Lorentz-Minkowski space $\L^{N+1}$, is defined as the vector space $\R^{N+1}$ equipped with the symmetric bilinear form
$$(x,y)_{\L^{N+1}} := x_1y_1+\ldots+x_Ny_N-x_{N+1}y_{N+1},$$
where $x=(x_1,\ldots,x_{N+1}), y=(y_1,\ldots,y_{N+1})$. The  bilinear form $(\cdot, \cdot )_{\L^{N+1}}$ is a non-degenerate bilinear form of index one (see \cite[Sect. A]{Spivak}), where the index of a bilinear form on a real vector space is defined as the largest dimension of a negative definite subspace. The modulus of $x \in \L^{N+1}$ is defined as $\|x\|_{\L^{N+1}}:=|(x,x)|_{\L^{N+1}}^{1/2}$.
\begin{definition}
We say that a vector $x \in \L^{N+1}$ is spacelike if $(x,x)_{\L^{N+1}}>0$. We say that a hypersurface $M\subset \L^{N+1}$ is spacelike if for any $p \in M$ the restriction of the metric $(\cdot,\cdot)_{\L^{N+1}}$ to the tangent space $T_pM$ is positive definite. In particular $M$ has a Riemannian structure. 
\end{definition}

We fix the notation and recall the following facts (for more details see also \cite[Sect. 2]{BS}): the indices $i,j$ have the range $1,\ldots,N$, while the indices $\mathcal{I},\mathcal{J}$ have the range $1,\ldots,N+1$, $\{e_{\mathcal{I}}\}$ denotes the natural basis of $\L^{N+1}$. In particular we have $(e_{\mathcal{I}}, e_\mathcal{J})=0$ if $\mathcal{I}\neq \mathcal{J}$, $(e_{i},e_{i})_{\L^{N+1}}=1$ and  $(e_{N+1},e_{N+1})_{\L^{N+1}}=-1$. Let $u \in C^2(\RN)$ be a strictly spacelike function and let $M:=\{(x,u(x)) \in \LN; \ x\in\RN \}$ be the vertical graph associated to $u$. We use $(x_1,\ldots,x_N)$ as coordinates on $M$ and $X_i=e_i+u_ie_{N+1}$ is a base of tangent vectors for $T_{p}M$, $p=(x,u(x))$. The induced metric on $M$ is given by $g=(g_{ij})_{ij=1,\ldots,N}$, where $g_{ij}=(X_i, X_j)_{\L^{N+1}}=\delta_{ij}-u_iu_j$, and $det (g)=1-|\nabla u|^2$. In particular, $M$ is a spacelike hypersurface if and only if $u$ is a strictly spacelike function. 

Let us set $v:=\sqrt{1-|\nabla u|^2}$. Denoting by $\nu$ the upward normal to $M$, $(\nu,\nu)_{\L^{N+1}}=-1$, it is easy see that $\nu=\frac{1}{v}(\nabla u, 1)=\sum_{i=1}^N \nu_i e_i + \nu_{N+1} e_{N+1}$, where $\nu_i=\frac{1}{v}u_i$ and $\nu_{N+1}=\frac{1}{v}$. The second fundamental form of $M$ is given by $A_{ij}=(X_i,\nabla_{X_j} \nu)_{\L^{N+1}}=\frac{1} {v}u_{ij}$ and we set $\|A \|^2:=\sum_{i,j,k,l=1}^N g^{ij}g^{kl}A_{ik}A_{jl}=\frac{1}{v^2}g^{ij}g^{kl}u_{ik}u_{jl}$, where $g^{-1}=(g^{ij})_{ij=1,\ldots,N}$, $g^{ij}=\delta_{ij}+\nu_i \nu_j$,  is the inverse matrix of $g$. 

The mean curvature of $M$ is given by $H=\sum_{i,j=1}^N g^{ij}A_{ij}=\frac{1}{v} \sum_{i,j=1}^N g^{ij} u_{ij}$ (we refer to \cite{Lopez2014} for more details). We denote by $\delta=\operatorname{grad}_M$, $\operatorname{div}_M$, $\Delta_M$, respectively, the gradient, the divergence and the Laplace-Beltrami operators on $M$.  

Let $W \subset \L^{N+1}$ an open neighborhood of $M$, let $Y$ be a $C^1$ vector field on $W $, and $f \in C^1(W)$ be such that $\frac{\partial f}{\partial x_{N+1}}=0$. We have the following:
\beq\label{eq:Usefulform}
\begin{array}{lll}
\displaystyle \delta f = \sum_{\mathcal{I}=1}^{N+1} (\delta_{\mathcal{I}}f) e_{\mathcal{I}},\ \ \hbox{where}\  \delta_i f=\sum_{j=1}^N  g^{ij} \frac{\partial}{\partial x_j}f, \ \delta_{N+1} f=\frac{1}{v}\sum_{j=1}^N  \nu_i \frac{\partial}{\partial x_i}f,&&\\[6pt]
\displaystyle \|\delta f\|_{\LN}^2 = \sum_{i,j=1}^{N+1} g^{ij}  \frac{\partial f}{\partial x_i} \frac{\partial f}{\partial x_j} = |\nabla f|^2 + \sum_{i=1}^N \left(\nu_i  \frac{\partial f}{\partial x_i}\right)^2,&&\\[16pt]
\displaystyle \operatorname{div}_M Y = \sum_{i,j=1}^N g^{ij} (X_i, \nabla_{X_j} Y)_{\LN}.&&\\
\end{array}
\eeq
If in addition $f \in C^2(W)$ we have
\beq\label{eq:UsefulformBis}
\displaystyle \Delta_M f =  \operatorname{div}_M  \operatorname{grad}_M f = \sum_{i,j=1}^N g^{ij} \frac{\partial^2}{\partial x_i \partial x_j} f+ \sum_{i=1}^NH \nu_i \frac{\partial}{\partial x_i} f.
\eeq
As a consequence of  Stoke's theorem we recall Green's formula: let $G\subset M$ bounded with $\partial G$ of class $C^1$ and outer normal $\sigma$ in $M$ and let $f,g \in C^2(G)$, then
\beq\label{eq:Green1}
\int_{G} \left( g \Delta_M f - f \Delta_Mg\right) \ dA=\int_{\partial G} \left(g \delta f - f \delta g, \sigma\right)_{\L^{N+1}} \ d\mu,
\eeq
where $dA$ is the induced volume form on $M$, given by $dA=v dx$, being $dx$ the Lebesgue measure on $\R^N$, and $d\mu$ is the surface measure on $\partial G$. If $f \in C_c^1(M)$, $g \in C^1(M)$ we have also the following formulas for  the integration by parts
\beq\label{eq:integparts}
\begin{array}{lll}
\displaystyle \int_{M} f \delta_{N+1} g\ dA &=& \displaystyle-  \int_{M} g \delta_{N+1} f\ dA + \int_{M}\frac{1}{v} f g \rho \ dA, \\[12pt]
\displaystyle \int_{M} f \delta_{i} g\ dA &=&\displaystyle -  \int_{M} g \delta_{i} f\ dA + \int_{M}\nu_i f g \rho \ dA.
\end{array}
\eeq

 Let $x_0 \in \RN$, we define the Lorentz distance from $(x_0,u(x_0))$ as $$l(x,x_0):=[(x-x_0)^2-(u(x)-u(x_0))^2]^{1/2}.$$ Given $R>0$ we define the Lorentz ball of radius $R$ centered at $x_0$ as $$L_R(x_0):=\{(x,u(x)) \in M; \ l(x,x_0)<R\}$$ and its projection on $\R^N$ as  $K_R(x_0):=\{x \in \R^N; \ l(x,x_0)<R\}$. We will use also the simpler notations $l=l(x)$, $L_R$, $K_R$ whenever $x_0$ is fixed and there is no chance of confusion. Finally, setting $X_0:=(x_0,u(x_0))\in \LN$ and using the relations \eqref{eq:Usefulform}-\eqref{eq:UsefulformBis} it is elementary to verify that

\beq\label{eq:Usefulform2a}
\displaystyle \|\delta l\|_{\LN}^2 = 1 + l^{-2} (\nu, X-X_0)_{\LN}^2,
\eeq
\beq\label{eq:Usefulform2b}
\displaystyle \Delta_M\left(\frac{1}{2} l^2\right)=N + \left(H\nu, X-X_0\right)_{\LN}.
\eeq

\section{A gradient estimate}
In this section we prove a new gradient estimate for strictly spacelike classical solutions of \eqref{eq:BI} with smooth data lying in $L^q(\R^N)$. 

\begin{proposition}\label{prop:mainestimate}
 Let $N\geq3$, $q>2N$, let $x_0 \in \R^N$, $R >0$ and let $\rho \in L^q(\R^N)\cap C^1(\R^N)$. Let $u \in C^3(\RN)$ be a strictly spacelike solution of \eqref{eq:BI} and set
 $$v:=\sqrt{1-|\nabla u|^2}.$$ 
 
 Then there exist two positive constants $\gamma \in \left(0,\frac{1}{N}\right)$, $C>0$ both depending only on $N$ such that
 \begin{equation}\label{eqTesiProp2}
\begin{array}{lll}
\displaystyle  \omega_Nv^\gamma(x_0)\! & \! \geq\! & \!\displaystyle e^{-\gamma/4}R^{-N} \int_{K_{R}(x_0)} \!\!v^{\gamma+1} \ dx - c(\rho) R   \left(\frac{2}{q}\right)^{\frac{q-2}{2}}\!\int_0^R  \! s^{-\beta} \left[\frac{q}{2}\left(\omega_N\right)^{\frac{2}{q}} + \frac{c(\rho)}{1-\beta} s^{2-\beta}\right]^{\frac{q-2}{2}} \ ds \\[12pt]
 & & +\, \displaystyle C R^{2-N} e^{-\gamma/4} \int_{K_{R/2}(x_0)}\sum_{i,j=1}^N u_{ij}^2 \ dx,
\end{array}
\end{equation}
where $K_R(x_0)=\{x \in \R^N; \ [(x-x_0)^2-(u(x)-u(x_0))^2]^{1/2}<R \}$, $\omega_N$ is the volume of the unit ball in $\R^N$, $\beta=\frac{2N}{q}$ and $c(\rho)= \frac{9}{4}  |\rho|_{q}^2$.
\end{proposition}

\begin{proof}
 Let $x_0 \in \R^N$, $R >0$. Since $u \in C^3(\R^N)$ is a strictly spacelike function the vertical graph $M=\{(x,u(x)) \in \L^{N+1}; \ x\in \R^N\}$ is a spacelike hypersurface.  
Up to a translation, we can assume without loss of generality that $X_0=(x_0,u(x_0))=(0,0)$. We denote by $X:=(x,u(x)) \in \L^{N+1}$ the position vector, and, respectively, by $L_R$, $K_R$ the Lorentz ball and its projection on $\RN$ centered at the origin. The starting point of our proof is \cite[(2.17)]{BS}. For the sake of completeness we  recall the main steps of the proof. Let $s \in (0,R)$, $f \in C^1(M)$ and $g=\frac{1}{2}(s^2-l^2)$. Since $\partial M=\emptyset$, then the outer normal of $\partial L_s$ in $M$ is given by $\sigma=\frac{\delta l}{\|\delta l\|_{\L^{N+1}}}$. Applying Green's formula \eqref{eq:Green1} with $G=L_s$ we get
$$\int_{L_s} \left[\frac{1}{2}(s^2-l^2) \Delta_M f + f \Delta_M\left(\frac{1}{2}l^2\right)\right] \ dA=\int_{\partial L_s} \left(\frac{1}{2}(s^2-l^2) \delta f + f \delta\left(\frac{1}{2} l^2\right), \frac{\delta l}{\|\delta l\|_{\L^{N+1}}} \right)_{\L^{N+1}} \ d\mu.$$

Now, using the relations \eqref{eq:Usefulform2a}-\eqref{eq:Usefulform2b} (with $X_0=(0,0)$, $H=-\rho$ because $u$ solves \eqref{eq:BI}) and recalling that $s^2-l^2=0$ on $\partial L_s$, by elementary computations we obtain
\beq\label{eqGF}
 \int_{L_s} \left[Nf + \frac{1}{2}(s^2-l^2) \Delta_M f - f \rho (X,\nu)_{\LN}\right] \ dA=\int_{\partial L_s}  f l \|\delta l\|_{\L^{N+1}} \ d\mu.
 \eeq
We recall the following special case of Federer's coarea formula (see \cite[(2.14)]{BS}, \cite[Theorem 3.2.12]{FED}), which is
\beq\label{eqGF2}
D_s \left[ \int_{L_s} h \ dA\right] = \int_{\partial L_s} h \|\delta l\|_{\LN}^{-1} \ d\mu, \ \ \forall h \in C^0(M).
\eeq
Since the integrand in the right-hand side of  \eqref{eqGF} can be rewritten as $ f l \|\delta l\|_{\L^{N+1}}^{2} \|\delta l\|_{\LN}^{-1}$, then using \eqref{eq:Usefulform2a} and applying \eqref{eqGF2} we have
\begin{eqnarray*}
s^{-N-1}\int_{\partial L_s}  f l \|\delta l\|_{\L^{N+1}} \ d\mu &=& s^{-N-1}\int_{\partial L_s}  f l \|\delta l\|_{\L^{N+1}}^{-1}  \ d\mu +  s^{-N-1}\int_{\partial L_s}  f l^{-1}(\nu, X)_{\LN}^2  \|\delta l\|_{\L^{N+1}}^{-1} \ d\mu\\
&=& s^{-N}\int_{\partial L_s}  f  \|\delta l\|_{\L^{N+1}}^{-1}  \ d\mu +  \int_{\partial L_s}  f l^{-N-2}(X, \nu)_{\LN}^2 \|\delta l\|_{\L^{N+1}}^{-1} \ d\mu\\
&=& s^{-N} D_s\left[\int_{ L_s}  f   \ dA\right] + D_s\left[\int_{ L_s}  f l^{-N-2}(X, \nu)_{\LN}^2  \ dA\right]. 
\end{eqnarray*}
Therefore, dividing each side of \eqref{eqGF} by $s^{-N-1}$ and using the previous relation we obtain the following monotonicity formula
\beq\label{eq:monotform}
\begin{array}{lll}
\displaystyle D_s\left[s^{-N}  \int_{ L_s}  f   \ dA\right] &=& \displaystyle \int_{L_s} s^{-N-1} \left(\frac{1}{2}(s^2-l^2) \Delta_M f - f \rho (X,\nu)_{\LN}\right) \ dA\\[8pt] 
 &&\displaystyle  -D_s\left[\int_{ L_s}  f l^{-N-2}(X, \nu)_{\LN}^2  \ dA\right].
\end{array}
\eeq 
Let $\gamma$ be a positive number to be determined later. By direct computation it holds that
\beq\label{eq:monotform2}
\begin{array}{lll}
\displaystyle \Delta_M v^\gamma \!&\!=\!&\! \displaystyle \gamma v^{\gamma-1} \Delta_M v + \gamma (\gamma - 1) v^{\gamma-2} \|\delta v\|_{\LN}^2\\[8pt]
\!&\!=\!&\! \displaystyle - \gamma v^{\gamma-2} \left[ \sum_{i,j=1}^N u_{ij}^2 - \gamma \left( \sum_{i=1}^N u_{ii}\right)^2 \!\! \!+ (1-\gamma) v \sum_{i=1}^N u_{ii} + v^2\rho^2 + (1-\gamma) \sum_{j=1}^N\left(\sum_{i=1}^N \nu_i u_{ij}\right)^2\right]\\[22pt]
\!&\! \!&\! \displaystyle + \gamma \delta_{N+1} \left(v^{\gamma+1}\rho \right).
\end{array}
\eeq 
Using the well known inequalities $ \left( \sum_{i=1}^N u_{ii}\right)^2 \leq N \sum_{i,j=1}^N u_{ij}^2$ and $ab\leq \e a^2 + \frac{1}{4\e} b^2$ for any $\e>0$, by elementary algebraic considerations we can find $\gamma \in (0,\frac{1}{N})$ and $C>0$ depending only on $N$ such that
\beq\label{eq:monotform3}
 \Delta_M v^\gamma \leq - C v^{\gamma-2} \left[\sum_{i,j=1}^N u_{ij}^2 + \sum_{j=1}^N\left(\sum_{i=1}^N \nu_i u_{ij}\right)^2 \right] + \frac{1}{4} v^\gamma \rho^2 + \gamma \delta_{N+1} \left(v^{\gamma+1}\rho \right).
\eeq

Combining \eqref{eq:monotform} and \eqref{eq:monotform3} we infer that
\begin{equation}\label{monotform}
\begin{array}{lll}
\displaystyle D_s\left\{s^{-N} \int_{L_s} v^\gamma \ dA\right\} 
&\leq& \displaystyle -\, C \int_{L_s} \frac{1}{2} s^{-N-1} (s^2-l^2)\left(\sum_{i,j=1}^N u_{ij}^2 + \sum_{j=1}^N\left(\sum_{i=1}^N \nu_i u_{ij}\right)^2\right) v^{\gamma-2} \ dA\\[12pt]
& & + \displaystyle \int_{L_s} \frac{1}{2} s^{-N-1} (s^2-l^2)\left(\frac{1}{4} v^\gamma \rho^2 + \gamma \delta_{N+1}(v^{\gamma+1}\rho)\right) \ dA\\[12pt]
& & +  \displaystyle \int_{L_s}  s^{-N-1} \rho ( X, \nu)_{\L^{N+1}} v^{\gamma} \ dA -  D_s\left\{\int_{L_s} ( X, \nu)_{\L^{N+1}}^2 l^{-N-2}v^\gamma \ dA\right\},
\end{array}
\end{equation}
which, in our setting, corresponds to \cite[(2.17)]{BS}.  Now, since $\frac{\partial}{\partial x_{N+1}} l =0$, we deduce from the first relation of \eqref{eq:Usefulform} that $ v \delta_{N+1} l =\sum_{i=1}^N \nu_i \frac{\partial}{\partial x_{i}} l $ and it follows that
\beq\label{eq:218BS}
|v \delta_{N+1} l | \leq \|\delta l \|_{\LN}.
\eeq

From \eqref{monotform}, integrating by parts (using the first relation in \eqref{eq:integparts}) and using \eqref{eq:Usefulform2a}, \eqref{eq:218BS} we deduce that

\begin{equation}\label{eq1monform}
\begin{array}{lll}
 \displaystyle \int_{L_s}  (s^2-l^2) \gamma \delta_{N+1}(v^{\gamma+1}\rho) \ dA &=& \displaystyle  \int_{L_s} (s^2-l^2)  v^\gamma \rho^2 \ dA + 2 \int_{L_s} v^{\gamma+1} \rho l \delta_{N+1}l \ dA  \\[12pt]
&\leq&  \displaystyle \int_{L_s} s \left( R^{-1} \|\delta l\|_{\L^{N+1}}^2 +  (s + R) \rho^2\right)v^\gamma \ dA.
\end{array}
\end{equation}
In view of \eqref{eq:Usefulform2a}, using the elementary estimate $\rho ( X, \nu)_{\L^{N+1}}  \leq \frac{1}{2}l^2\rho^2+\frac{1}{2}l^{-2} ( X, \nu)_{\L^{N+1}}^2$, and since $0<l<s$ in $L_s$, which implies that $s^{-N-1} (s^2-l^2) < s^{-N}2s$, we deduce from \eqref{monotform} and \eqref{eq1monform} that
\begin{equation}\label{eq2monform}
\begin{array}{lll}
\displaystyle D_s\left\{s^{-N} \int_{L_s} v^\gamma \ dA\right\} 
&\leq& \displaystyle -\, C \int_{L_s} \frac{1}{2} s^{-N-1} (s^2-l^2)\left(\sum_{i,j=1}^N u_{ij}^ 2  + \sum_{j=1}^N\left(\sum_{i=1}^N \nu_i u_{ij}\right)^2\right) v^{\gamma-2} \ dA\\
& & + \displaystyle \frac{1}{4} s^{-N+1} \int_{L_s}  v^\gamma \rho^2 \ dA + \frac{\gamma}{2} R^{-1} s^{-N} \int_{L_s} v^\gamma \ dA \\ & &  + \displaystyle  \frac{\gamma}{2} R^{-1} s^{-N} \int_{L_s} l^{-2} ( X, \nu)_{\L^{N+1}}^2 v^\gamma \ dA 
+  \displaystyle  \gamma R s^{-N} \int_{L_s}  v^\gamma \rho^2 \ dA \\ & &  + \displaystyle \frac{1}{2}s^{-N+1} \int_{L_s}  v^\gamma \rho^2 \ dA+  \frac{1}{2} s^{-N-1} \int_{L_s} l^{-2} ( X, \nu)_{\L^{N+1}}^2  v^{\gamma} \ dA \\
&& - \displaystyle  D_s\left\{\int_{L_s} ( X, \nu)_{\L^{N+1}}^2 l^{-N-2}v^\gamma \ dA\right\}.
\end{array}
\end{equation}
By reorganizing the terms and using the trivial estimates 
$$\frac{1}{4}s+\gamma R + s<\frac{9}{4}R, \ \ \frac{1}{2}\gamma R^{-1} + \frac{1}{2}s^{-1}<s^{-1},$$ 
we obtain
\begin{equation}\label{eq3monform}
\begin{array}{rl}
\displaystyle D_s\left\{s^{-N} \int_{L_s} v^\gamma \ dA\right\}
\leq& \displaystyle -\, C \int_{L_s} \frac{1}{2} s^{-N-1} (s^2-l^2)\left(\sum_{i,j=1}^N u_{ij}^2 + \sum_{j=1}^N\left(\sum_{i=1}^N \nu_i u_{ij}\right)^2\right) v^{\gamma-2} \ dA\\[16pt]
  & + \displaystyle \frac{9}{4}R s^{-N} \int_{L_s}  v^\gamma \rho^2 \ dA + \frac{\gamma}{2} R^{-1} s^{-N} \int_{L_s} v^\gamma \ dA\\
  &  +   \displaystyle s^{-N-1} \int_{L_s} l^{-2} ( X, \nu)_{\L^{N+1}}^2 v^\gamma \ dA - \displaystyle  D_s\left\{\int_{L_s} ( X, \nu)_{\L^{N+1}}^2 l^{-N-2}v^\gamma \ dA\right\}.
\end{array}
\end{equation}
We now estimate the term $s^{-N} \int_{L_s}  v^\gamma \rho^2 \ dA$. Set $p:=\frac{q}{q-2}$ (the conjugate exponent of $\frac{q}{2}$) and $\beta:=\frac{2N}{q} \in (0,1)$. Since $\rho^2 \in L^{\frac{q}{2}}(\R^N)$, applying H\"older's inequality and taking into account that $v\leq 1$, we infer that
\begin{align}\label{eq4monform}
\displaystyle s^{-N} \int_{L_s}  v^\gamma \rho^2 \ dA =  s^{-N} \int_{K_s}  v^{\gamma+1} \rho^2 \ dx  
&\leq \displaystyle  s^{-\frac{2N}{q}} \left( s^{-N}\int_{K_s}  v^{(\gamma+1)p} \ dx \right)^{\frac{1}{p}}  \left(\int_{K_s} |\rho|^q \ dx\right)^{\frac{2}{q}}\\[8pt]
&\leq \displaystyle  s^{-\beta} \left( s^{-N} \int_{L_s}  v^{\gamma} \ dA \right)^{\frac{1}{p}}  \left(\int_{\R^N} |\rho|^q \ dx\right)^{\frac{2}{q}}.\nonumber
\end{align}
Next, set
$$\psi(s):=s^{-N} \int_{L_s} v^\gamma \ dA.$$
We deduce from \eqref{eq3monform} and \eqref{eq4monform} that
\begin{align}\label{eq5monform}
\displaystyle\psi^\prime(s)\leq & \displaystyle -\, C \int_{L_s} \frac{1}{2} s^{-N-1} (s^2-l^2)\left(\sum_{i,j=1}^N u_{ij}^2 + \sum_{j=1}^N\left(\sum_{i=1}^N \nu_i u_{ij}\right)^2\right) v^{\gamma-2} \ dA \\
& + \displaystyle \frac{9}{4}R  \left(\int_{\R^N} |\rho|^q \ dx\right)^{\frac{2}{q}} s^{-\beta}  \psi^{\frac{1}{p}}(s) + \frac{\gamma}{2} R^{-1}\psi(s) +   s^{-N-1} \int_{L_s} l^{-2} ( X, \nu)_{\L^{N+1}}^2 v^\gamma \ dA \nonumber  \\
& - \displaystyle  D_s\left\{\int_{L_s} ( X, \nu)_{\L^{N+1}}^2 l^{-N-2}v^\gamma \ dA\right\}. \nonumber 
\end{align}
We rewrite \eqref{eq5monform} in a more convenient way
\begin{multline}\label{eq6monform}
\displaystyle\psi^\prime(s) - \frac{\gamma}{2} R^{-1}\psi(s) -  \displaystyle c(\rho) R   s^{-\beta}  \psi^{\frac{1}{p}}(s) \\
 \leq \displaystyle -\, C \int_{L_s} \frac{1}{2} s^{-N-1} (s^2-l^2)\left(\sum_{i,j=1}^N u_{ij}^2 + \sum_{j=1}^N\left(\sum_{i=1}^N \nu_i u_{ij}\right)^2\right) v^{\gamma-2} \ dA\\
+   \displaystyle  s^{-N-1} \int_{L_s} l^{-2} ( X, \nu)_{\L^{N+1}}^2 v^\gamma \ dA - D_s\left\{\int_{L_s} ( X, \nu)_{\L^{N+1}}^2 l^{-N-2}v^\gamma \ dA\right\},
\end{multline}
where $c(\rho):= \frac{9}{4}  \left(\int_{\R^N} |\rho|^q \ dx\right)^{\frac{2}{q}}$. As in \cite{BS}, using the co-area formula we rewrite the last two terms of \eqref{eq6monform} as $-T^\prime(s)$, where
$$ T(s)= \int_{L_s}\left[ \left(1-\frac{1}{N}\right) l^{-N-2}+ \frac{1}{N}s^{-N}l^{-2}\right] ( X, \nu)_{\L^{N+1}}^2 v^\gamma \ dA.$$
Multiplying each side of \eqref{eq6monform} by the integrating factor $F(s)=e^{-\frac{\gamma}{4} R^{-1} s}$, recalling that $\frac{1}{p} + \frac{2}{q}= 1$ and setting $\U(s):=e^{-\frac{\gamma}{4} R^{-1} s} \psi(s)$, we obtain
\begin{multline}\label{eq7monform}
\displaystyle\U^\prime(s)- c(\rho) R  e^{-\frac{\gamma}{2q} R^{-1} s} s^{-\beta}\ \U^{\frac{1}{p}}(s) \\
 \leq  \displaystyle -\, C F(s) \int_{L_s} \frac{1}{2} s^{-N-1} (s^2-l^2)\left(\sum_{i,j=1}^N u_{ij}^2 + \sum_{j=1}^N\left(\sum_{i=1}^N \nu_i u_{ij}\right)^2\right) v^{\gamma-2} \ dA - F(s)T^\prime(s).
\end{multline}
Let us recall that, introducing the step function $\lambda:\R\to[0,1]$ defined by $\lambda(t)=0$ if $t<0$ and $\lambda(t)=1$ if $t\geq0$, we have
\beq\label{eq7monformstep}
 \int_{L_s} h \ dA= \int_M \lambda(s-l) h \ dA,\ \forall h \in L^1(M).
\eeq
Moreover, since $v^\gamma$ is continuous at $x_0=0$ and $l$ approximates the geodesic distance in $M$ for $|x|$ small (see for instance \cite[Sect. 2]{BS})), it follows that 
\beq\label{eq7monformGeo}
\mathcal{U}(s) \to \omega_N v^{\gamma}(0),\ \ \hbox{as}\ s \to 0^+.
\eeq
Therefore, from \eqref{eq7monform}, integrating between $0$ and $R$, using \eqref{eq7monformstep}, Fubini's theorem, and taking \eqref{eq7monformGeo} into account, we deduce that
\begin{multline}\label{eq8monform}
\displaystyle\U(R)- \omega_Nv^\gamma(0) - c(\rho) R  \int_0^R  s^{-\beta} e^{-\frac{\gamma}{2q} R^{-1} s}\ \U^{\frac{1}{p}}(s) \ ds \\
 \leq \displaystyle -\, C  \int_{L_R} \left[\left(\sum_{i,j=1}^N u_{ij}^2 + \sum_{j=1}^N\left(\sum_{i=1}^N \nu_i u_{ij}\right)^2 v^{\gamma-2} \right) \cdot \left( \int_0^R  \frac{1}{2} s^{-N-1} (s^2-l^2)F(s)  \lambda(s-l) ds\right)\right] \ dA\\
\quad \quad -  \displaystyle \int_0^R  F(s)T^\prime(s) \ ds.\hfill 
\end{multline}
We observe that $ \int_0^R  F(s)T^\prime(s) \ ds\geq 0$. Indeed, since $M$ is $C^2$ and strictly spacelike we have $(X,\nu)_{\L^{N+1}}=O(|x|^2)$ as $|x|\to 0$ and thus 
$$\int_{L_s} ( X,\nu)_{\L^{N+1}}^2 l^{-N-2}\ dA \to 0, \ \ \ \int_{L_s} ( X,\nu)_{\L^{N+1}}^2 s^{-N}l^{-2}\ dA \to 0, \ \ \ \hbox{as}\ s\to 0^+.$$ Then, taking into account that $T(s)\geq 0$, $F^\prime(s)\leq 0$ and integrating by parts we have 
\begin{equation}\label{eq9monform}
 \int_0^R  F(s)T^\prime(s) \ ds =F(R)T(R)- 0 - \int_0^R F^\prime(s)T(s) \ ds\geq0.
\end{equation}
At the end from \eqref{eq8monform}, \eqref{eq9monform}, using the fact that $F(s)\geq F(R)$ for all $s \in [0,R]$, and computing the integral $$ \int_0^R  \frac{1}{2} s^{-N-1} (s^2-l^2)\lambda(s-l) ds,$$ 
we obtain
\begin{multline}\label{eq10monform}
\displaystyle\U(R)- \omega_Nv^\gamma(0) - c(\rho) R  \int_0^R  s^{-\beta} e^{-\frac{\gamma}{2q} R^{-1} s}\ \U^{\frac{1}{p}}(s) \ ds \\
 \leq \displaystyle -\, C F(R)  \int_{L_R} S_R(l)\left(\sum_{i,j=1}^N u_{ij}^2 + \sum_{j=1}^N\left(\sum_{i=1}^N \nu_i u_{ij}\right)^2 v^{\gamma-2} \right) \ dA,
\end{multline}
where $$S_R(l)=\frac{1}{N(N-2)} l^{2-N}+\frac{1}{2N}l^2R^{-N}-\frac{1}{2(N-2)}R^{2-N}.$$
Now, observing that $S_R(l)>0$ for $0<l<R$, we get 
\begin{equation*}
\U(R)- \omega_Nv^\gamma(0) - c(\rho) R  \int_0^R  s^{-\beta} e^{-\frac{\gamma}{2q} R^{-1} s}\ \U^{\frac{1}{p}}(s) \ ds\leq 0,
\end{equation*}
and in  particular
$$\U(R)\leq \omega_N + c(\rho) R  \int_0^R  s^{-\beta} \ \U^{\frac{1}{p}}(s) \ ds.$$
Since $R > 0$ is arbitrary, we can rewrite the previous relation as
\begin{equation}\label{eq12monform}
\U(t)\leq \omega_N +   \int_0^t  c(\rho) t s^{-\beta} \ \U^{\frac{1}{p}}(s) \ ds\ \ \ \forall t >0.
\end{equation}
Applying Theorem \ref{variantGrom} with $C_0=\omega_N$, $\Psi(s)=C_1s^{-\beta}$, $g(k)=k^{\frac{1}{p}}$, where $C_1=c(\rho)t$, we get that
\begin{equation}\label{eq12amonform}
\U(t)\leq \left(\frac{2}{q}\right)^{\frac{q}{2}} \left[\frac{q}{2}\left(\omega_N\right)^{\frac{2}{q}} + \frac{c(\rho)}{1-\beta} t^{2-\beta}\right]^{\frac{q}{2}}\ \ \forall t >0,
\end{equation}
where we have taken into account that $\frac{1}{p}=\frac{q-2}{q}$, $\Phi(l)=\frac{q}{2}l^{\frac{2}{q}}$ and Remark \ref{rem:VarGrom}.
 Now, coming back to \eqref{eq10monform}, observing that $S_R(l)>c(N)R^{2-N}$ for $0<l<R/2$, for some positive constant $C(N)$ depending only on $N$, recalling that $dA=v dx$, $v^{\gamma-1}\geq 1$ (because $\gamma-1<0$), $F(R)=e^{-\gamma/4}$, $e^{-\frac{\gamma}{2q} R^{-1} s}\leq 1$ for $s\in (0,R]$, $\frac{1}{p}=\frac{q-2}{q}$ and using \eqref{eq12amonform}, we obtain

\begin{multline}\label{eq13monform}
\displaystyle C(N) R^{2-N} e^{-\gamma/4} \int_{K_{R/2}}\sum_{i,j=1}^N u_{ij}^2 \ dx \\
 \leq \displaystyle \omega_Nv^\gamma(0)- \U(R) + c(\rho) R   \left(\frac{2}{q}\right)^{\frac{q-2}{2}}\int_0^R  s^{-\beta} \left[\frac{q}{2}\left(\omega_N\right)^{\frac{2}{q}} + \frac{ c(\rho)}{1-\beta} s^{2-\beta}\right]^{\frac{q-2}{2}} \ ds.
\end{multline}
Finally, using the definition of $\U$ we have
\begin{equation}\label{eq14monform}
\begin{array}{ll}
\displaystyle  \omega_Nv^\gamma(0)\geq & \displaystyle e^{-\gamma/4}R^{-N} \int_{K_{R}} v^{\gamma+1} \ dx - c(\rho) R   \left(\frac{2}{q}\right)^{\frac{q-2}{2}}\int_0^R  s^{-\beta} \left[\frac{q}{2}\left(\omega_N\right)^{\frac{2}{q}}+ \frac{ c(\rho)}{1-\beta} s^{2-\beta}\right]^{\frac{q-2}{2}} \ ds \\[12pt]
 &+\, \displaystyle C(N) R^{2-N} e^{-\gamma/4} \int_{K_{R/2}}\sum_{i,j=1}^N u_{ij}^2 \ dx,
\end{array}
\end{equation}
which, up to a translation, gives the desired relation and the proof is complete.

\end{proof}

\section{$W^{2,2}_{loc}$ regularity of the minimizer}
\begin{proof}[Proof of Theorem \ref{mainteo}]
Let $u_\rho \in \X$ be the minimizer of $I_\rho$ and let $(u_n)_n \subset \X$ be the coresponding sequence of minimizers of $I_{\rho_n}$, where $(\rho_{n})_{n}$ is a standard sequence of mollifications of $\rho$. By construction and in view of\cite[Remark 3.4, Remark 5.5]{BDP} we know that  $(u_n)_n$ is made of smooth strictly spacelike solutions of \eqref{eq:BI} with data $\rho_{n}$, and $(u_n)_n$ converges to $u_\rho$ weakly in $\X$, and uniformly in $\RN$.
 
We claim that  $\|u_n\|_\X$ is uniformly bounded. From the fact that $0 \in \X$, $I_{\rho_n}(0)=0$ and since $u_n \in \X$ is a minimizer for $I_{\rho_n}$, it follows that 
\beq\label{eqminimizerutil}
I_{\rho_n}(u_n)\leq 0, \ \ \ \forall  n \in \N.
\eeq 
 Moreover, let us recall the following elementary algebraic inequality 
\beq\label{ineq}
\frac 12 t\le 1-\sqrt{1-t}\le t, \ \ \ \forall t\in [0,1].
\eeq
We first complete the proof of the claim assuming that $m \in\mathopen]1,2_*]$. In that case, from \eqref{eqminimizerutil}, \eqref{ineq}, H\"older's inequality and Sobolev's inequality, we deduce that 
\begin{equation}\label{eq3termainteo}
\begin{array}{lll}
\displaystyle \frac{1}{2}|\n u_n|^2_2
&\le&\displaystyle \irn \Big( 1- \sqrt{1-|\n u_n|^2}\Big) \ dx\\[12pt]
&\le&\displaystyle \langle \rho_n ,u_n \rangle \leq \displaystyle |\rho_n|_{m} | u_n|_{m^\prime} \leq\displaystyle c(N,m)|\rho_n|_{m} |\n u_n|_{2}^{2 \frac{(N+1)m-N}{mN}}
\end{array}  
\end{equation}
where $m^\prime=\frac{m}{m-1}$ is the conjugate exponent of $m$, $c(N,m):=c_0\left(N, \frac{mN}{(N+1)m-N}\right)$, $c_0(N,k)$ is the Sobolev constant for the embedding of $\mathcal{D}^{1,k}(\RN)$ into $L^{k^*}(\R^N)$, $k \in [2,N\mathclose[$. In fact, for any $\phi \in \X$ we have $|\n \phi | \leq 1$ in $\RN$ and thus $|\n \phi| \in L^k(\RN)$ for all $k\geq 2$. Therefore, by Sobolev's inequality $\phi \in L^{k^*}(\RN)$ for $k \in [2,N\mathclose[$ and
\beq\label{eq:ApplSobineq} 
|\phi|_{k^*} \leq c_0(N,k) |\n \phi|_k \leq c_0(N,k) |\n \phi|_2^{2/k}.
\eeq
Finally, it is elementary to verify that for a given $m \in\mathopen]1,2_*]$ it is always possible to find $k \in [2,N\mathclose[$ such that $k^*=m^\prime$. By direct computation it holds that $k=\frac{mN}{(N+1)m-N}$ and we are done.

Observe that $|\rho_n|_{L^m(\RN)}  \leq |\rho|_{L^m(\RN)} $ by the properties of the convolution and the standard mollification sequence. Hence, from \eqref{eq3termainteo} and 
by elementary computations, we immediately deduce that 
\begin{equation}\label{eq3quatermainteo}
\|u_n\|_{\X} \leq \left(2 c(N,m) |\rho|_{m}\right)^{\frac{mN}{2(N-m)}},
\end{equation} 
which completes the proof of the claim when $m \in\mathopen]1,2_*]$.

We now consider the claim in the case $m=1$. By Morrey-Sobolev's inequality, for a given $s>N$, there exists a positive constant $c$ depending only on $N$, $s$, which we denote for convenience by $c(N,1)$, such that
\beq\label{eq:MorreySobolev}
 |\phi|_\infty \leq c(N,1) \|\phi\|_{W^{1,s}(\RN)}, \ \forall \phi \in W^{1,s}(\RN),
\eeq
where $\|\phi\|_{W^{1,s}(\RN)}=\|\phi\|_s+\|\n \phi\|_s$ is the standard norm in $W^{1,s}(\RN)$. Arguing as before, using the fact that $|\nabla \phi| \leq 1$ in $\RN$ for any $\phi \in \X$ and Sobolev's inequality, we can show that $\X$ continuously embeds into $W^{1,s}(\RN)$. Indeed, for any $s>N$ we can always find $k \in\mathopen]\frac{N}{2}, N\mathclose[$ such that $k^*=s$, namely $k=\frac{Ns}{N+s}$, and by \eqref{eq:MorreySobolev}, \eqref{eq:ApplSobineq} we deduce that
\begin{equation}\label{eqcontembedd}
|\phi|_\infty \leq c(N,1) (c(N,s^\prime) |\n \phi|_2^{2 \frac{N+s}{Ns}} + |\n \phi|_2^{2/s}), \ \forall \phi \in \X.
\end{equation}
Therefore, fixing $s>N$ and arguing as in the proof of \eqref{eq3termainteo}, taking into account that $|\rho_n|_{1} \leq |\rho|_{1}$, we get
\begin{equation}\label{eq3termainteobis}
 |\n u_n|^2_2  \leq  2 c(N,1) (c(N,s^\prime)+1) |\rho|_{1} \max\left\{|\n u_n|_2^{2\frac{N+s}{Ns}}, |\n u_n|_2^{2/s} \right\}.
\end{equation}
This proves that the sequence $(u_n)_n$ is bounded in $\X$.

\medbreak

Now we complete the proof of the theorem. Let $x_0 \in \R^N$ and $R > 0$. Denote by $K^n_R(x_0)$ the Lorentz ball associated to $(u_n)_n$. Thanks to Proposition \ref{prop:mainestimate}, we have
 \begin{equation*}
\begin{array}{lll}
\displaystyle C R^{2-N} e^{-\gamma/4} \int_{K^n_{\frac{R}{2}}(x_0)}\sum_{i,j=1}^N (u_n)_{ij}^2 \ dx&\leq&\displaystyle  \omega_Nv_n^\gamma(x_0)- \displaystyle e^{-\gamma/4} R^{-N}\int_{K^n_{R}(x_0)} v_n^{\gamma+1} \ dx \\[12pt]
& & +\displaystyle  c(\rho_n) R   \left(\frac{2}{q}\right)^{\frac{q-2}{2}}\int_0^R  s^{-\beta} \left[\frac{q}{2}\omega_N^{\frac{2}{q}} + \frac{c(\rho_n)}{1-\beta} s^{2-\beta}\right]^{\frac{q-2}{2}} \ ds.
\end{array}
\end{equation*}
Recalling that $\gamma$ and $C$ depend only on $N$, $c(\rho_n)\leq c(\rho)$ and that $B_{\frac{R}{2}}(x_0)\subset K^n_{\frac{R}{2}}(x_0)$ for all $n \in \N$, we conclude that
$$ \int_{B_{\frac{R}{2}}(x_0)}\sum_{i,j=1}^N (u_n)_{ij}^2 \ dx \leq \tilde{C},$$
for some constant $\tilde{C}=\tilde{C}(N,q,\rho,R)>0$. From this, since $|\nabla u_n| \leq 1$ in $\R^N$ and since $(u_n)_n$ uniformly bounded, we deduce that $(u_n)_n$ is bounded in $W^{2,2}(B_{\frac{R}{2}}(x_0))$. Therefore $u_n\rightharpoonup \bar u$ in $W^{2,2}(B_{\frac{R}{2}}(x_0))$, for some $\bar u \in W^{2,2}(B_{\frac{R}{2}}(x_0))$, and as $u_\rho=\bar u$ in $B_{\frac{R}{2}}(x_0)$, it follows that $u_\rho \in W^{2,2}_{loc}(\R^N)$. The proof is then complete.
\end{proof}

\section{Validity of the Euler-Lagrange equation and $C^{1,\alpha}_{loc}$ regularity}
\begin{proof}[Proof of Theorem \ref{mainteo2}]
Let $q>2N$, $m \in [1, 2_*]$, $\rho\in L^q(\mathbb{R}^N)\cap L^{m}(\R^N)$ and let $u_\rho \in \X$ be the minimizer of $I_\rho$. As in the proof of Theorem \ref{mainteo} we consider the sequence $(u_n)_n \subset \X$ of minimizers for $I_{\rho_n}$, where $(\rho_n)_n$ is a standard sequence of mollifications of $\rho$, and we set $v_n:=\sqrt{1-|\nabla u_n|^2}$. Assume first that $m \in\mathopen]1, 2_*]$. We divide the proof in several steps.\\

 \noindent\textbf{Step 1:} Let $R>0$. There exists a constant $c_1=c_1(N,m,R,\rho)>0$, which is explicit, such that
for any $y \in \RN$ it holds 
  \begin{equation}\label{eq3bismainteo}
  \int_{B_R(y)} v_n^{-1} \ dx \leq c_1, \ \forall n \in \N.
 \end{equation}
\\
 In view of \cite[Remark 5.4]{BDP} we know that $u_n \in \X$ is weak solution of \eqref{eq:BI}, hence
 \begin{equation}\label{eq4mainteo}
\irn \frac{  |\n  u_n|^2}{ \sqrt{1- |\n u_n|^2}}\, dx
= \irn \rho_n u_n.
\end{equation}
Arguing as in the proof of \eqref{eq3termainteo} and using \eqref{eq3quatermainteo}, we have
 \begin{equation}\label{eq5mainteo}
\left|\irn \rho_n u_n \ dx\right| \leq  |\rho_n|_{m} |u_n|_{m^\prime} \leq 2^{\frac{(N+1)m-N}{N-m}} \left(c(N,m)|\rho|_{m}\right)^{m^*},
\end{equation}
where $m^*=\frac{Nm}{N-m}$, $c(N,m)$ is a positive constant depending only on $N$ and $m$. Therefore, for any fixed $y \in \R^N$, $R>0$, we deduce from \eqref{eq4mainteo} and  \eqref{eq5mainteo} that

 \begin{equation}\label{eq6mainteo}
\int_{B_R(y)}  \frac{  |\n  u_n|^2}{ \sqrt{1- |\n u_n|^2}}\, dx\leq  \irn \frac{  |\n  u_n|^2}{ \sqrt{1- |\n u_n|^2}}\, dx\leq 2^{\frac{(N+1)m-N}{N-m}} \left(c(N,m)|\rho|_{m}\right)^{m^*}.
\end{equation}
Now, noticing that 
$$\int_{B_R(y)}  \frac{  |\n  u_n|^2}{ \sqrt{1- |\n u_n|^2}}\, dx = \int_{B_R(y)}  \frac{  |\n  u_n|^2-1}{ \sqrt{1- |\n u_n|^2}}\, dx+  \int_{B_R(y)}  \frac{  1}{ \sqrt{1- |\n u_n|^2}}\, dx $$
we get
$$\int_{B_R(y)}  \frac{  1}{ \sqrt{1- |\n u_n|^2}}\, dx\leq \omega_N R^N + 2^{\frac{(N+1)m-N}{N-m}} \left(c(N,m)|\rho|_{m}\right)^{m^*},$$
which gives \eqref{eq3bismainteo}, with $c_1:=\omega_N R^N + 2^{\frac{(N+1)m-N}{N-m}}  \left(c(N,m)|\rho|_{m}\right)^{m^*}$.\\

\noindent\textbf{Step 2:} Let $R>0$. There exists a constant $c_2=c_2(N,m,R,\rho)>0$, which is explicit, such that for any $y \in \R^N$ it holds
 \begin{equation}\label{eq3mainteo}
 \int_{B_R(y)} v_n^{\gamma+1} \ dx \geq c_2 R^{N(\gamma+2)}, \ \forall n \in \N,
  \end{equation}
where $\gamma \in (0,\frac{1}{N})$ is the positive constant (depending only on $N$) given by Proposition \ref{prop:mainestimate}.

\smallbreak
  
\noindent Fixing $y \in \R^N$ and $R \in (0,1]$ and recalling that $u_n$ is strictly spacelike we can write
$$\omega_N R^N = \int_{B_R(y)} 1 \ dx =  \int_{B_R(y)} v_n^{\frac{\gamma+1}{\gamma+2}} v_n^{-\frac{\gamma+1}{\gamma+2}} \ dx.$$
Noticing that $v_n^{\frac{\gamma+1}{\gamma+2}} \in L^{\gamma+2}(B_R(y))$, as $(\gamma+2)^\prime=\frac{\gamma+2}{\gamma+1}$, we obtain from H\"older's inequality and Step 1 that
\begin{equation*}
\begin{array}{lll}
\displaystyle \omega_N R^N  &\leq &\displaystyle\left( \int_{B_R(y)} v_n^{\gamma+1} \ dx\right)^{\frac{1}{\gamma+2}} \left(\int_{B_R(y)} v_n^{-1} \ dx\right)^{\frac{\gamma+1}{\gamma+2}}\\
& \leq&\displaystyle  \left( \int_{B_R(y)} v_n^{\gamma+1} \ dx\right)^{\frac{1}{\gamma+2}}  \left(\omega_N R^N+ 2^{\frac{(N+1)m-N}{N-m}} \left(c(N,m)|\rho|_{m}\right)^{m^*}\right)^{\frac{\gamma+1}{\gamma+2}}
\end{array}
\end{equation*}
which readily implies \eqref{eq3mainteo} with $$c_2(N,m,R,\rho)= \frac{\omega_N^{\gamma+2}}{\left[\omega_N R^N+2^{\frac{(N+1)m-N}{N-m}} \left(c(N,m)|\rho|_{m}\right)^{m^*}\right]^{1+\gamma}}.$$

\medbreak

\noindent\textbf{Step 3:} There exists a constant $c_3=c_3(N,m,q)>0$ such that if $|\rho|_q+|\rho|_{m}\leq c_3$ there exists $\delta \in (0,1)$ depending only on $N$, $m$, $q$ and $\rho$ such that
$$| \nabla u_n|_\infty\leq 1 -\delta, \ \forall n \in \N.$$
Let $x_0 \in \R^N$ and $R >0$. Denote by $K^n_R(x_0)$ the Lorentz ball associated to the sequence $(u_n)_n$. Thanks to Proposition \ref{prop:mainestimate} we have

\begin{equation}
\displaystyle  \omega_Nv_n^\gamma(x_0)\geq  \displaystyle e^{-\gamma/4} R^{-N} \int_{K^n_{R}(x_0)} v_n^{\gamma+1} \ dx - c(\rho_n) R   \left(\frac{2}{q}\right)^{\frac{q-2}{2}}\int_0^R  s^{-\beta} \left[\frac{q}{2}\left(\omega_N\right)^{\frac{2}{q}} + \frac{c(\rho_n)}{1-\beta} s^{2-\beta}\right]^{\frac{q-2}{2}} \ ds.
\end{equation}
Recalling that $\gamma$ depends only on $N$, $c(\rho_n)\leq c(\rho)$ and that $B_{\frac{R}{2}}(x_0)\subset K^n_{\frac{R}{2}}(x_0)$ for all $n \in \N$, we infer from Step 2 and elementary inequalities that

\begin{equation*}
\begin{array}{lll}
\displaystyle  \omega_Nv_n^\gamma(x_0)&\geq & \displaystyle e^{-\gamma/4}  \displaystyle  \frac{\omega_N^{\gamma+2}}{\left[\omega_N R^N+2^{\frac{(N+1)m-N}{N-m}} \left(c(N,m)|\rho|_{m}\right)^{m^*}\right]^{1+\gamma}}R^{N(\gamma+1)}  \\[20pt]
& & -\displaystyle 2^{\frac{q-4}{2}}c(\rho)\left[ \omega_N^{\frac{q-2}{p}} \frac{R^{2-\beta}}{1-\beta} +  c(\rho)^{\frac{q-2}{2}}  \left(\frac{2}{q}\right)^{\frac{q-2}{2}} \frac{R^{q-N}}{(1-\beta)^{\frac{q-2}{2}}(q-N-1)}\right].
\end{array}
\end{equation*}
Choosing $R=1$ and dividing by $\omega_N$, we obtain
 \begin{equation*}
\begin{array}{lll}
\displaystyle v_n^\gamma(x_0)&\geq & \displaystyle e^{-\gamma/4}  \displaystyle  \frac{\omega_N^{\gamma+1}}{\left[\omega_N+2^{\frac{(N+1)m-N}{N-m}} \left(c(N,m)|\rho|_{m}\right)^{m^*}\right]^{1+\gamma}} \\[20pt]
&&-2^{\frac{q-4}{2}}c(\rho)\left[ \frac{ \omega_N^{-\frac{2}{q}}}{1-\beta} +  c(\rho)^{\frac{q-2}{2}}  \left(\frac{2}{q}\right)^{\frac{q-2}{2}} \frac{\omega_N^{-1}}{(1-\beta)^{\frac{q-2}{2}}(q-N-1)}\right].
\end{array}
\end{equation*}
Therefore, if $|\rho|_q+|\rho|_{m}\leq c_3$, for some sufficiently small constant $c_3=c_3(N,m,q)$ depending only on $N$, $m$ and $q$, we deduce that 
$$ v_n(x_0) \geq \delta,$$
for some $\delta=\delta(N,m,q,\rho) \in (0,1)$. Recalling the definition of $v_n$, we obtain that $$|\nabla u_n(x_0)|\leq 1 - \delta^2.$$ 
Since $\delta$ does not depend on $n$ nor on $x_0$, and as $x_0$ is arbitrary, the conclusion follows.\\

\noindent\textbf{Step 4:} If $|\rho|_q+|\rho|_{m}\leq c_3$, where $c_3$ is the constant given in Step 3, then the minimizer $u_\rho$ is a weak solution to \ref{eq:BI}.

\smallbreak

Let $\varphi \in C^{\infty}_0(\R^N)$. From the proof of Theorem \ref{mainteo} we deduce that $u_n \rightharpoonup u_\rho$ in $W^{2,2}_{loc}(\R^N)$ and then, up to a subsequence, $\nabla u_n \to \nabla u_\rho$ a.e. in compact subsets of $\R^N$. Thanks to Step 3 we have that $\frac{1}{\sqrt{1-|\nabla u_n|^2}} \leq (2\delta-\delta^2)^{-1/2}$ and thus by Lebesgue's dominated converge, we get
$$\lim_{n \to + \infty} \irn \frac{\n u_n \cdot \n \varphi}{\sqrt{1-|\nabla u_n|^2}}\, dx
=\irn \frac{\n u_\rho \cdot \n \varphi}{\sqrt{1-|\nabla u_\rho|^2}}\, dx.$$
 On the other hand, since $\rho_n \to \rho$ in $L^q(\R^N)$, we have
$$\lim_{n \to + \infty} \irn \frac{\n u_n \cdot \n \varphi}{\sqrt{1-|\nabla u_n|^2}}\, dx
=\lim_{n \to + \infty} \langle \rho_n, \varphi\rangle = \langle \rho, \varphi \rangle,$$
and thus $u_\rho$ is a weak solution of \eqref{eq:BI}.\\

\noindent\textbf{Step 5:} If $|\rho|_q+|\rho|_{m}\leq c_3$, where $c_3$ is the constant given in Step 3, then $u_\rho \in C^{1,\alpha}_{loc}(\R^N)$, for some $\alpha \in (0,1)$ and it is strictly spacelike.

\smallbreak

Let $\epsilon \in (0,1)$ and let $\eta_\e \in C^\infty([0,+\infty))$ such that $r\eta_\e \in C^\infty([0,+\infty))$, $r \mapsto \eta_\e(r)r$ is increasing and satisfies
\begin{equation*}
 \eta_\e(r)r =\begin{cases} r &\hbox{for} \ r< 1-\e,\\[6pt]
       1-\frac{1}{2}\e & \hbox{for} \ r> 1-\frac{1}{2}\e.
       \end{cases}
\end{equation*}       
Let us consider the function $a_\e:\R^N\to\R^N$ given by
\begin{equation}\label{defOpreg}
a_\e(z):=\frac{z}{\sqrt{1-\eta^2_\e(|z|)|z|^2}}.
\end{equation}
Clearly $a_\e \in C^1(\R^N,\R^N)$ and it is easy to check (see Step 6) that $a_\e$ satisfies the growth and ellipticity conditions \cite[(1.2)]{KUM}, with $p=2$, $s=0$, $\nu=1$ and for some constant $L>1$ depending only on $\e$. Now, choosing $\e \in (0,\delta)$, where $\delta>0$ is given by Step 3, we have by construction that $\eta_\e(|\nabla u_n|)|\nabla u_n|=|\nabla u_n|$ for all $n$. Therefore, $(u_n)_n$ is a sequence of smooth weak solutions of
$$-\operatorname{div} \left(a_\e(\nabla u_n)\right)=\rho_n \ \hbox{in} \ \R^N.$$
Now, fixing a bounded domain $\Omega \subset \R^N$, and a ball $B_r \subset \Omega$, we can apply \cite[Theorem 1.4]{KUM} with $\omega\equiv0$ because $a_\e$ depends only on the variable $z$ (see \cite[(1.15)]{KUM} for the definition of $\omega)$ and $s=0$. Hence, as $\omega\equiv 0$, for $\alpha \in (0,\alpha_M)$, where $\alpha_M \in(0,1)$ is the universal H\"older exponent associated to the homogeneous equation $-\operatorname{div} \left(a_\e(\nabla u)\right)=0$ (see \cite[Sect. 1.1]{KUM}), we get that for any $x,y \in B_{r/4}$,

\begin{equation}\label{eq:estimateDu}
|Du_n(x)-Du_n(y)|\leq c \left[{\bf W}^{\rho_n}_{\frac{1-\alpha}{2},2}(x,r)+{\bf W}^{\rho_n}_{\frac{1-\alpha}{2},2}(y,r)\right]|x-y|^\alpha + c \fint_{B_r}|Du_n| \ dx \cdot \left(\frac{|x-y|}{r}\right)^\alpha,
\end{equation}
for any $n \in \N$, where the constant $c$ depends only on $N$, $\nu$, $L$ and $diam(\Omega)$ and 
$${\bf W}^{\rho_n}_{\frac{1-\alpha}{2},2}(x,r)={\bf I}^{\rho_n}_{1-\alpha}(x,r)= \int_0^r \frac{\int_{B_t(x)}|\rho_n(\xi)| \ d\xi}{t^{N-(1-\alpha)}} \frac{dt}{t}$$
is the standard (truncated) Wolff potential (which coincides with the truncated Riesz potential for $p=2$, see \cite[Sect. 1.1]{KUM}).
By using H\"older's inequality we have
$$\int_{B_t(x)}|\rho_n(\xi)| \ d\xi \leq (\omega_N t)^{\frac{N(q-1)}{q}} |\rho_n|_{q,B_t(x)} \leq c(N) t^{\frac{N(q-1)}{q}} |\rho|_p,$$
and thus
\begin{eqnarray*}
{\bf I}^{\rho_n}_{1-\alpha}(x,r) &\leq & c(N) |\rho|_q \int_0^r t^{-\alpha-\frac{N}{q}} \ dt.
\end{eqnarray*}
Therefore, since $\frac{N}{q}<1$ we can choose $\alpha$ sufficiently small so that $\alpha+\frac{N}{q}<1$ and we conclude that
$${\bf I}^{\rho_n}_{1-\alpha}(x,r) \leq c(N,q,r)  |\rho|_q.$$  Hence from \eqref{eq:estimateDu}, recalling that $|u_n|$ and $|\nabla u_n|$ are uniformly bounded, we get that $\|u_n\|_{C^{1,\alpha}(B_{r/4})}$ is uniformly bounded. Finally, from well known pre-\, Compactness results (see e.g. \cite[Lemma 6.36]{GT}), up to a subsequence, for $0<\alpha^\prime<\alpha$ it holds that $u_n \to \tilde u$ in $C^{1,\alpha^\prime}(\overline{B_{r/8}})$ for some $\tilde u \in C^{1,\alpha^\prime}(\overline{B_{r/8}})$. Now, since $\tilde u=u_\rho$ and $|\nabla\tilde u|\leq 1-\epsilon$ in $\overline{B_{r/8}}$, the thesis follows and the proof of the step is complete.
 
 \medbreak
 
 \noindent\textbf{Step 6:} The function $a_\e$ defined in \eqref{defOpreg} verifies the conditions \cite[(1.2)]{KUM} with $s=0$, $p=2$, $\nu=1$ and $L>1$ depending only on $\e$.
 
 \smallbreak
 
By construction $a_\e \in C^1(\R^N,\R^N)$, and, using the notations of \cite{KUM}, we denote by $\partial a_\e$ the Jacobian matrix of $a_\e$, by $|\cdot|$ both the standard euclidean norm in $\R^N$ and the euclidean matrix norm in $\mathcal{M}_N(\R)$, and by $(\cdot,\cdot)$ the euclidean scalar product in $\R^N$. Defining $f_\e:[0,+\infty) \to [0,+\infty)$ as $f_\e(r):=\frac{1}{\sqrt{1-\eta_\e^2(r)r^2}}$, and denoting by $f_\e^\prime$ its derivative, by direct computation we have:
$$\partial a_\e(z)=\left(f_\e^\prime(|z|)\frac{z_i z_j}{|z|}+f_\e(|z|) \delta_{ij}\right)_{i,j=1,\ldots,N}.$$
By construction, we have $1\leq f_\e \leq \frac{1}{\sqrt{\e-\frac{\e^2}{4}}}$ and thus for any $z \in \R^N$,
$$|a_\e(z)|=|f_\e(|z|) z|\leq \frac{1}{\sqrt{\e-\frac{\e^2}{4}}} |z|.$$ 
Moreover, we have $f_\e^\prime(r)=\frac{r\eta_\e(r)(\eta_\e^\prime(r) r +\eta_\e(r))}{\sqrt{(1-\eta_\e^2(r)r^2)^3}}\geq 0$ because $r\mapsto\eta_\e(r) r$ is increasing, and again by construction, we have $f^\prime_\e(r)=0$ for $r>1-\frac{\e}{2}$.
From this, it follows that for any $z \in \R^N$,
\begin{eqnarray*}
|\partial a_\e(z)| |z| &\leq&  {f_\e^\prime(|z|)} |(z_iz_j)_{i,j=1,\ldots,n}| + f_\e(|z|) |z|\\[6pt]
&\leq& \max_{|z|\leq 1-\frac{\e}{2}}\left(\frac{f_\e^\prime(|z|)}{|z|} |(z_iz_j)_{i,j=1,\ldots,n}|\right)|z| + \frac{1}{\sqrt{\e-\frac{\e^2}{4}}} |z|.
\end{eqnarray*}
On the other hand, there exists a constant $L>1$ depending only on $\e$ such that
$$|a_\e(z)|+|\partial a_\e(z)| |z| \leq L |z|,$$
and thus the growth condition in \cite[(1.2)]{KUM} is satisfied with $s=0$, $ p=2$. For the ellipticity condition, recalling that $f_\e\geq 1$ and $f^\prime_\e\geq 0$, we get that for any $\lambda \in \R^N$,
\begin{eqnarray*}
\left(\partial a_\e \lambda ,\lambda \right)&=& \sum_{i,j=1}^N \left[f_\e^\prime(|z|)\frac{z_i z_j}{|z|} \lambda_i\lambda_j + f_\e(|z|) \delta_{ij}\lambda_i\lambda_j\right]\\
&=&  \frac{f_\e^\prime(|z|)}{|z|} (\lambda,z)^2 + f_\e(|z|)|\lambda|^2 \geq|\lambda|^2,
\end{eqnarray*}
and this completes the proof of this step and thus of the theorem when $m \in\mathopen]1,2_*]$.\\

We prove now the result for $m=1$. The scheme of the proof is the same as in the previous case, but slight modifications to Step 1, Step 2 and Step 3 are needed. From \eqref{eq3termainteobis} it is clear that  
\beq\label{eq1casem1}
|\n u_n|^{2-F(n)}_2  \leq  2 c(N,1) (c(N,s^\prime)+1) |\rho|_{1},
\eeq
where $$F(n)= \begin{cases}  \frac{2 (N+s)}{Ns} & \hbox{if}\ |\n u_n|_2> 1,\\  \frac{2 }{s} & \hbox{if}\ |\n u_n|_2\leq 1.\end{cases}$$
Now observe that from \eqref{eq1casem1}, by choosing $|\rho|_1$ sufficiently small so that  $$2 c(N,1) (c(N,s^\prime)+1) |\rho|_{1}\leq 1,$$ we get $|\n u_n|_2\geq 1$ and thus $F(n)= \frac{2 }{s}$ for all $n \in \N$. Hence, setting $\bar c(N,s):=c(N,1) (c(N,s^\prime)+1)$, by elementary computations, we have
\beq\label{eq1casem2}
|\n u_n|_2  \leq  (2 \bar c(N,s) |\rho|_{1})^{\frac{s}{2(s-1)}}. 
\eeq
Recalling \eqref{eqcontembedd}, we immediately deduce that
\beq\label{eq1casem3}
|u_n|_\infty  \leq  2^{\frac{1}{s-1}} \bar c(N,s)^{\frac{s}{s-1}} |\rho|_{1}^{\frac{1}{s-1}}.
\eeq
Therefore, \eqref{eq5mainteo} in Step 1 is now substituted by

 \begin{equation}\label{eq5mainteocasem1}
\left|\irn \rho_n u_n \ dx\right| \leq  |\rho_n|_{1} |u_n|_{\infty} \leq  2^{\frac{1}{s-1}} \left(\bar c(N,s) |\rho|_{1}\right)^{\frac{s}{s-1}},
\end{equation}
and Step 1 follows with $c_1(N,s,R,\rho)=\omega_N R^N+2^{\frac{1}{s-1}} \left(\bar c(N,s) |\rho|_{1}\right)^{\frac{s}{s-1}}$. The proof of Step 2 is the same with $c_2(N,s,R,\rho)= \frac{\omega_N^{\gamma+2}}{\left[\omega_N R^N+2^{\frac{1}{s-1}} \left(\bar c(N,s) |\rho|_{1}\right)^{\frac{s}{s-1}}\right]^{1+\gamma}}$ and in Step 3 (taking $R=1$) we have
 \begin{equation*}
\begin{array}{lll}
\displaystyle v_n^\gamma(x_0)&\geq & \displaystyle e^{-\gamma/4}  \displaystyle  \frac{\omega_N^{\gamma+1}}{\left[\omega_N+2^{\frac{1}{s-1}} \left(\bar c(N,s) |\rho|_{1}\right)^{\frac{s}{s-1}}\right]^{1+\gamma}} \\[20pt]
&&-2^{\frac{q-4}{2}}c(\rho)\left[ \frac{ \omega_N^{-\frac{2}{q}}}{1-\beta} +  c(\rho)^{\frac{q-2}{2}}  \left(\frac{2}{q}\right)^{\frac{q-2}{2}} \frac{\omega_N^{-1}}{(1-\beta)^{\frac{q-2}{2}}(q-N-1)}\right],
\end{array}
\end{equation*}
and thus if $|\rho|_q+|\rho|_{1}\leq c_3$, for some sufficiently small constant $c_3=c_3(N,s,q)$ we deduce that 
$$ v_n(x_0) \geq \delta,$$
for some $\delta=\delta(N,s,q,\rho) \in (0,1)$. The remaining steps are similar to the case $m \in\mathopen]1,2_*]$. The proof is then complete.
\end{proof}

\begin{remark}
We point out that when $\rho$ is bounded, given an arbitrary bounded domain $\Omega$ with smooth boundary, denoting by $\overline{xy}$ the segment joining two points $x$, $y$, and by
\begin{equation}
\label{eq:setK}
K:=\left\{\overline{xy}\subset\Omega;\  x,y\in\partial\Omega, x\neq y, |u_\rho(x)-u_\rho(y)|=|x-y|\right\},
\end{equation}
the set of light rays associated to $u_\rho$, then from \cite[Corollary 4.2]{BS} we have that $u_\rho$ is a strictly spacelike weak solution of \eqref{eq:BI} in $\Omega\setminus K$. Now, if $K\neq\emptyset$, by invoking \cite[Theorem 3.2]{BS} it follows that any light ray extends to the whole $\RN$ obtaining a contradiction. Unfortunately, when $\rho$ is not bounded it is not possible to replicate entirely the proof of \cite[Theorem 3.2]{BS} and it could happen that $u_\rho$ contains a finite segment of a light ray having as endpoints two singularities of $\rho$. This question is still open and we are investigating on it. 
\end{remark}

\appendix
\section{A variant of Gronwall's Lemma}
The following theorem is a generalization of a result due to Bihari (see \cite{BH}).

\begin{theorem}\label{variantGrom}
Let $T>0$ and let $\U:[0,T] \to [0,+\infty\mathclose[$ be a continuous function such that
\begin{equation}\label{eq:prophyp}
 \U(t) \leq C_0 + \int_0^t\Psi(s) g(\U(s)) \ ds\ \ \forall t \in [0, T],
\end{equation}
where $C_0>0$, $\Psi:(0,T]\to \R^+$ is continuous, $\Psi \in L^1(0,T)$, $g:[0,+\infty)\to [0,+\infty)$ is continuous, (strictly) monotone increasing, such that $\frac{1}{g} \in L^1(0,l)$ for any $l>0$ and set $\displaystyle \Phi(l):=\int_0^l\frac{1}{g(k)} \ dk$. Then
$$\U(t)\leq \Phi^{-1}\left(\Phi(C_0)+ \int_0^t\Psi(s) \ ds\right) \ \forall t \in [0,T].$$
\end{theorem}

\begin{proof}
Let us consider the function $y:[0,T] \to \R$ defined by
$$y(t):= \int_0^t \Psi(s) g(\U(s)) \ ds.$$

Since $\Psi \in L^1(0,T)$ and since $g(\U(s))$ is bounded in $[0,T]$ we have $y(0)=0$. Thanks to \eqref{eq:prophyp}, and being $g$ monotone increasing, we have
\begin{equation}\label{eq:propgrom}
y^\prime(t)=g(\U(t))\Psi(t) \leq g(C_0+y(t)) \Psi(t), \ \forall t\in(0,T].
\end{equation}
Let us observe that since $y(t)\geq0$, $C_0>0$ and being $g\geq0$ strictly increasing we have $g(C_0+y(t))\geq g(C_0)>g(0)\geq 0$ for all $t\in[0,T]$. Therefore dividing by $g(C_0+y(t))$ each side of \eqref{eq:propgrom} and integrating on $(0,t)$ we get that
$$ \int_0^t \frac{y^\prime(s)}{g(C_0+y(s))} \ ds \leq \int_0^t \Psi(s) \ ds, \ \forall t \in (0,T].$$
Changing variable in the first integral, and recalling the definition of $\Phi$ and that $y(0)=0$, we have
\begin{equation}\label{eq:propgrom2}
\Phi(y(t)+C_0)-\Phi(C_0)=\int^{y(t)+C_0}_{C_0} \frac{1}{g(k)} \ dk  \leq \int_0^t \Psi(s) \ ds.
\end{equation}
Now, since $\Phi$ is increasing, recalling the definition of $y(t)$ and using \eqref{eq:prophyp} we obtain
\begin{equation}\label{eq:propgrom3}
\Phi(y(t)+C_0)\geq \Phi(\U(t)).
\end{equation}
At the end, being $\Phi$  invertible and combining \eqref{eq:propgrom2}, \eqref{eq:propgrom3} we deduce that 
$$\U(t)\leq \Phi^{-1}\left(\Phi(C_0)+ \int_0^t\Psi(s) \ ds\right) \ \forall t \in (0,T],$$
and from \eqref{eq:prophyp} it is obvious that $\U(0)\leq C_0$. The proof is then complete.
\end{proof}

\begin{remark}\label{rem:VarGrom}
Let $\beta, \gamma \in (0,1)$ and let $C_0>0$, $C_1\geq 0$. If $\U:[0,T] \to [0,+\infty\mathclose[$ is continuous and verifies
$$ \U(t) \leq C_0 + \int_0^t C_1 s^{-\beta} \left(\U(s)\right)^{\gamma} \ ds\ \ \forall t \in [0, T],$$
then applying Theorem \ref{variantGrom} with $\Psi(s)=C_1 s^{-\beta}$, $g(k)=k^{\gamma}$, since $\Phi(l)=\frac{1}{1-\gamma}l^{1-\gamma}$ by elementary computations we get that $$\U(t) \leq (1-\gamma)^{\frac{1}{1-\gamma}}\left[\frac{C_0^{1-\gamma}}{1-\gamma}+C_1 \frac{t^{1-\beta}}{1-\beta}\right]^{\frac{1}{1-\gamma}}\ \ \forall t \in [0,T].$$
A similar result was already known in the literature (see \cite[Theorem 7.1]{SATA}), but Theorem \ref{variantGrom} applies to a much wider class of nonlinearities $g$.
\end{remark}

\end{document}